\documentclass[11pt,reqno,sumlimits]{amsart}

\usepackage{amssymb, amscd, amsmath, epsfig, mathtools}
\usepackage[utf8]{inputenc}
\usepackage{amsthm}
\usepackage{enumerate}
\usepackage{xcolor}
\usepackage{scalerel}
\usepackage{soul}
\usepackage{tikz-cd}

\usepackage[margin=1.0in]{geometry}

\newtheorem{theorem}{Theorem}[section]
\newtheorem*{theorem*}{Theorem}
\newtheorem*{proposition*}{Proposition}
\newtheorem{definition}{Definition}[section]

\newtheorem{lemma}{Lemma}[section]
\newtheorem{proposition}{Proposition}[section]
\theoremstyle{definition}
\newtheorem{remark}{Remark}[section]

\newcommand{\R}{\mathbb R}

\newcommand{\calC}{\mathcal C}
\newcommand{\calL}{\mathcal L}
\newcommand{\dvol}{ d\text{Vol}_{g}}

\begin{document}

\title[Prescribed Gaussian and Geodesic Curvatures]{Kazdan-Warner Problem on Compact Riemann Surfaces with Smooth Boundary}
\author[J. Xu]{Jie Xu}
\address{
111 Cummington Mall, Department of Mathematics and Statistics, Boston University, Boston, MA, USA, 02215}
\email{xujie@bu.edu}
\address{
600 Dunyu Road, Institute for Theoretical Sciences, Westlake University, Hangzhou, Zhejiang, China, 310030}
\email{xujie67@westlake.edu.cn}

\date{}	
						
\begin{abstract} In this article, we show that (i) any smooth function on compact Riemann surface with non-empty smooth boundary $ (M, \partial M, g) $ can be realized as a Gaussian curvature function; (ii) any smooth function on $ \partial M $ can be realized as a geodesic curvature function for some metric $ \tilde{g} \in [g] $. The essential steps are the existence results of Brezis-Merle type equations $ -\Delta_{g} u + Au = K e^{2u} \; {\rm in} \; M $ and $ \frac{\partial u}{\partial \nu} + \kappa u = \sigma e^{u} \; {\rm on} \; \partial M $ with given functions $ K, \sigma $ and some constants $ A, \kappa $. In addition, we rely on the extension of the uniformization theorem given by Osgood, Phillips and Sarnak.
\end{abstract}

\maketitle

\section{Introduction}
The Kazdan-Warner problem has long history on closed Riemann surface and was comprehensively studied by Kazdan and Warner in 1970s, \cite{KW2, KW3, KW}. The significance of the conformal deformation in the Kazdan-Warner problem links the geometric question with the solvability of some nonlinear elliptic PDEs. Therefore the Kazdan-Warner problem is of primary interest in the area of geometric analysis. It is natural to ask the question of prescribing Gaussian or geodesic curvature functions on compact Riemann surfaces with non-empty smooth boundary, $ (M, \partial M, g), n = \dim M = 2 $. Precisely speaking, we consider prescribing Gaussian and geodesic curvature problem for metrics within a conformal class $ [g] $. Denoting the conformal change to be $ \tilde{g} = e^{2u} g $ for some $ u \in \calC^{\infty}(M) $, the problem of realizing the given functions $ K, \sigma $ as Gaussian and geodesic curvatures of $ \tilde{g} $ is reduced to the existence of smooth solutions of the following PDE
\begin{equation}\label{intro:eqn1}
-\Delta_{g} u + K_{g} = K e^{2u} \; {\rm in} \; M, \frac{\partial u}{\partial \nu} + \sigma_{g} = \sigma e^{u} \; {\rm on} \; \partial M.
\end{equation} 
Here $ -\Delta_{g} $ is the positive definite Laplace-Beltrami operator whose coefficients have smooth extension to $ \partial M $, $ K_{g} $ and $ \sigma_{g} $ are Gaussian and geodesic curvatures of $ g $, respectively. The prescribed geodesic curvature function $ \sigma $ is originally defined on $ \partial M $. Throughout this article, we sometimes use the same notation $ \sigma \in \calC^{\infty}(M) $, which means some smooth extension of the original function $ \sigma \in \calC^{\infty}(\partial M) $. It is well-known that the Gaussian and geodesic curvatures are dominated by the remarkable Gauss-Bonnet theorem, and hence by the Euler characteristic $ \chi(M) $:
\begin{equation}\label{intro:eqn2}
\int_{M} K_{g} \dvol + \int_{\partial M} \sigma_{g} dA_{g} = 2\pi \chi(M).
\end{equation}
Here $ \dvol $ is the Riemannian density on $ M $, $ dA_{g} $ is the volume form on $ \partial M $, both for the metric $ g $. The Gauss-Bonnet formula (\ref{intro:eqn2}) is invariant under conformal change, i.e. (\ref{intro:eqn2}) holds with the volume forms for $ \tilde{g} $ if we replace $ K_{g} $ by $ K $, $ \sigma_{g} $ by $ \sigma $, provided that $ K $ and $ \sigma $ can be realized as associated curvature functions with respect to $ \tilde{g} = e^{2u} g $. 

As necessary conditions, it follows that either $ K $ or $ \sigma $ must be positive somewhere, provided that $ \chi(M) > 0 $, or either $ K $ or $ \sigma $ must be negative somewhere if $ \chi(M) < 0 $, etc. When $ \chi(M) = 0 $, possible cases are $ K > 0 $, $ \sigma < 0 $, $ K $ changes sign and $ \sigma \equiv 0 $, etc.

In this article, we are interested in the sufficient conditions of prescribing Gaussian and geodesic curvature problems within conformal class $ [g] $ of $ (M, \partial M, g) $ with all possible Euler characteristics. We have very recent results in \cite{XU9, XU8} for negative Euler characteristic by monotone iteration scheme, and for zero Euler characteristic with minimal boundary by variational method. Most known results in this topics are mainly in four directions: (i) prescribing both Gaussian and geodesic curvatures for negative or zero Euler characteristics using variational method, or for negative Euler characteristics with flow methods, see e.g. \cite{CMR, HO}, ; (ii) prescribing Gaussian curvatures for minimal boundary or prescribing geodesic curvatures for scalar-flat metrics, see \cite{CV}, etc.; (iii) prescribing Gaussian curvatures on closed Riemann surfaces, see \cite{AC, STRUWE}, etc. for Nirenberg's problem, see \cite{KW2} for both zero and negative Euler characteristics. In addition, many existence results were given bydirect methods of calculus of variations, blow-up analysis and Liouville theorems, see e.g. \cite{AmLiMa, BiMaRi, ChXuYa, CMR, CrRu, DjMaAh, XuZh}.

Our main result states that any smooth function can be realized as either a Gaussian curvature function or a geodesic curvature function for some metric within the conformal class $ [g] $, meanwhile some associated geodesic/Gaussian curvature functions are assigned simultaneously.
\begin{theorem*}\label{intro:thm0}
Let $ (M, \partial M, g) $ be a compact Riemann surface with non-empty smooth boundary.
\begin{enumerate}[(i)]
\item Any function $ K \in \calC^{\infty}(M) $ can be realized as a Gaussian curvature function of some metric $ \tilde{g} \in [g] $;
\item Any function $ \sigma \in \calC^{\infty}(\partial M) $ can be realized as a geodesic curvature function of some metric $ \tilde{\bar{g}} \in [g] $.
\end{enumerate}
\end{theorem*}
\medskip

It is quite different from the results on closed Riemann surfaces, which are more restrictive. We gain the flexibility here due to the existence of the boundary; for instance, the Gauss-Bonnet formula (\ref{intro:eqn2}) says that both the Gaussian and geodesic curvatures contribute to the sign of the Euler characteristics. The main tools include the model cases on Riemann surfaces with boundary--the uniformization theorem on compact Riemann surfaces with boundary. In particular, we rely on the remarkable extension of the classical uniformization theorem that was given by Osgood, Phillips and Sarnak \cite{OsPhSa}:
\begin{theorem}\label{intro:thm1}\cite[Introduction]{OsPhSa}
Let $ (M, \partial M, g) $ be a compact Riemann surface with non-empty smooth boundary. Let $ [g] $ be the conformal class of $ g $. Then
\begin{enumerate}[(i)]
\item There exists a metric $ \tilde{g}_{1} \in [g] $ such that $ K_{\tilde{g}_{1}} $ is some constant and $ \partial M $ of zero geodesic curvature with respect to $ \tilde{g}_{1} $;
\item There exists a metric $ \tilde{g}_{2} \in [g] $ such that $ \sigma_{\tilde{g}} $ is some constant and $ M $ is flat with respect to $ \tilde{g}_{2} $.
\end{enumerate}
\end{theorem} 
The first result is shown in Theorem \ref{ARB:thm1}, and the second result is shown in Theorem \ref{ARB:thm2}.
\medskip

In addition, we rely on the existence result on Brezis-Merle type equations
\begin{equation}\label{intro:eqn3}
-\Delta_{g} u + Au = K e^{2u} \; {\rm in} \; M, \frac{\partial u}{\partial \nu} + \kappa u = c\sigma e^{u} \; {\rm on} \; \partial M.
\end{equation}
for given $ K, \sigma \in \calC^{\infty}(M) $. Here $ A, \kappa, c $ are some given nonnegative constants such that either $ A = 0, \kappa > 0 $ or $ A > 0, \kappa = 0 $. Due to the positivity of either $ A $ or $ \kappa $, we can apply many tools like the maximum principle, etc. to get more information of the solutions of (\ref{intro:eqn3}). In conclusion, the following proposition plays a central role in showing the main result.
\begin{proposition*}\label{intro:prop1}
Let $ (M, \partial M, g) $ be a compact manifold with non-empty smooth boundary, $ n = \dim M \geqslant 2 $. Let $ K \in \calC^{\infty}(M), \sigma \in \calC^{\infty}(\partial M) $ be given functions.
\begin{enumerate}[(i)]
\item If $ A  = 0 $ and $ \kappa > 0 $, then there exist some positive constants $ C_{0}, C_{1} $ such that (\ref{intro:eqn1}) has a smooth solution $ u $ with given $ K $ and $ \sigma $ such that $ K \leqslant C_{0} $ on $ M $, and $ c \leqslant C_{1} $;
\item If $ A  > 0 $ and $ \kappa = 0 $, then there exist some positive constants $ C_{2}, C_{3} $ such that (\ref{intro:eqn1}) has a smooth solution $ u $ with given $ K $ and $ \sigma $ such that $ K \leqslant C_{2} $ on $ M $, and $ c \leqslant C_{3} $.
\end{enumerate}
\end{proposition*}
The first conclusion is given in Theorem \ref{BM:thm1}, and the second conclusion is given in Theorem \ref{BM:thm2}. The result above can be easily extended to the case $ A > 0, \kappa > 0 $ but we do not need that version in this article.
\medskip

As another application of the existence results of Brezis-Merle type equations, we can prescribe non-positive Gaussian curvature function and arbitrary geodesic curvature function for metrics within a conformal class  provided that $ \chi(M) = 0 $, the topology is important here as we cannot get similar results for manifolds with $ \chi(M) > 0 $, see Theorem \ref{zero:thm1}. 
\medskip

This article is organized as follows. \S1 is the introduction and list of main results. \S2 contains all necessary definitions and results that are used in later sections. \S3 introduces the existence of solutions of Brezis-Merle type equations in Theorem \ref{BM:thm1} and Theorem \ref{BM:thm2}. \S4 is the main section of this article, which proves the main result as shown in Theorem \ref{ARB:thm1} and Theorem \ref{ARB:thm2}. In \S5, we show the results of prescribing non-positive Gaussian and arbitrary geodesic curvature functions in Theorem \ref{zero:thm1}, provided that $ \chi(M) = 0 $. We give necessary conditions for prescribing Gaussian and geodesic functions with $ \chi(M) < 0 $. In \S7, we show that the sign condition of prescribing Gaussian and geodesic curvature for $ \chi(M) > 0 $ case is only restricted by the topology--the Gauss-Bonnet theorem--in Propostion \ref{ps:prop1}. Similar results are given in Proposition \ref{ps:prop2} for $ \chi(M) < 0 $.
\section{The Preliminaries}
In this section, we give essential definitions and results from differential geometry and the theory of elliptic PDEs on manifolds. Throughout this article, we denote smooth functions $ f $ by $ f \in \calC^{\infty}(M) $ and $ kth $ continuously differentiable functions $ f' $ by $ f' \in \calC^{k}(M) $; $ k = 0 $ implies that $ f' $ is a continuous function. We also denote $ \nu $ to be the unit outward normal vector field along $ \partial M $ on compact manifolds $ (M, \partial M, g) $ with non-empty smooth boundary. We assume the standard background of linear elliptic theory, including Lax-Milgram theorem \cite[\S6]{Lax}, the Sobolev embedding theorem \cite[\S2]{Aubin}, the standard maximum principle \cite[\S5]{T}, the trace theorem \cite[Ch.~4]{T}, Schauder estimates \cite{GT}, etc. We will give results of monotone iteration scheme for nonlinear elliptic PDEs with nonlinear oblique boundary conditions, a Stampacchia type maximum principle for some boundary value problems and a $ W^{s, p} $-type elliptic regularity for elliptic PDEs with oblique boundary conditions.

We begin with the definition of integer-ordered Sobolev spaces on compact manifolds (possibly with boundary). Although in most situation, $ (M, \partial M, g) $ represents compact Riemann surfaces with non-empty boundary, we state our results for general $ n $-dimensional compact manifolds with boundary if possible.
\begin{definition}\label{HL:def1} Let $ (M, g) $ be a closed Riemannian $ n $-manifold, and $ (M, \partial M, g) $ be a compact Riemannian $n$-manifold with non-empty smooth boundary, with volume densities $\dvol$ on $ M $ and $ dS_{g} $ on $ \partial M $. Let $u$ be a real valued function. Let $ \langle v,w \rangle_g$ and $ |v|_g = \langle v,v \rangle_g^{1/2} $ denote the inner product and norm  with respect to $g$. 

(i) 
For $1 \leqslant p < \infty $, we define the Lebesgue spaces on $ M $ to be
\begin{align*}
\mathcal{L}^{p}(M, g)\ &{\rm is\ the\ completion\ of}\ \left\{ u \in \calC^{\infty}(M) : \Vert u\Vert_{\calL^{p}(M, g)}^p :=\int_{M} \left\lvert u \right\rvert^{p} \dvol < \infty \right\}.
\end{align*}

(ii) For $\nabla u$  the Levi-Civita connection of $g$, 
and for $ u \in \calC^{\infty}(\Omega) $ or $ u \in \calC^{\infty}(M) $,
\begin{equation}\label{pre:eqn1}
\lvert \nabla^{k} u \rvert_g^{2} := (\nabla^{\alpha_{1}} \dotso \nabla^{\alpha_{k}}u)( \nabla_{\alpha_{1}} \dotso \nabla_{\alpha_{k}} u).
\end{equation}
\noindent In particular, $ \lvert \nabla^{0} u \rvert^{2}_g = \lvert u \rvert^{2} $ and $ \lvert \nabla^{1} u \rvert^{2}_g = \lvert \nabla u \rvert_{g}^{2}.$\\

(iii) For $ s \in \mathbb{N}, 1 \leqslant p < \infty $, we define the $ (s, p) $-type Sobolev spaces on $ M $ to be
\begin{align}\label{pre:eqn2}
W^{s, p}(M, g) &= \left\{ u \in \mathcal{L}^{p}(M, g) : \lVert u \rVert_{W^{s, p}(M, g)}^{p} : = \sum_{j=0}^{s} \int_{M} \left\lvert \nabla^{j} u \right\rvert^{p}_g \dvol < \infty \right\}.
\end{align}
\noindent Here $ \lvert D^{j}u \rvert^{p} := \sum_{\lvert \alpha \rvert = j} \lvert \partial^{\alpha} u \rvert^{p} $ in the weak sense.  In particular, $ H^{s}(M, g) : = W^{s, 2}(M, g) $ are the usual Sobolev spaces. 

(iv) For $ s \in \mathbb{N}, 1 \leqslant p < \infty $, $ W_{0}^{s, p}(M, g) $ is defined to be the completion of $ \calC_{c}^{\infty}(M) $ with respect to the $ W^{s, p} $-norm. We define $ H_{0}^{s}(M, g) $ similarly.
\end{definition}
\medskip

Next we give a Stampacchia-type maximum principle, which was also used in \cite{KW2}, in which the prescribing Gaussian curvature problem on closed Riemann surfaces was studied.
\begin{proposition}\label{pre:prop1}
Let $ (M, \partial M, g) $ be a compact manifolds with non-empty smooth boundary. Let $ \kappa > 0 $ be some constant. If there exists a function $ u \in \calC^{2}(M) $ such that
\begin{equation}\label{pre:eqn3}
-\Delta_{g} u \leqslant 0, \frac{\partial u}{\partial \nu} + \kappa u \leqslant 0,
\end{equation}
Then $ u \leqslant 0 $ on $ M $.
\end{proposition}
\begin{proof}
For any $ v \in H^{1}(M, g) \cap \calC^{0}(M) $, with $ v \geqslant 0 $ on $ \partial M $, (\ref{pre:eqn3}) implies that
\begin{align*}
0 & \geqslant \int_{M} -\Delta_{g} u \cdot v \dvol = \int_{M} \nabla_{g} u \cdot \nabla_{g} v \dvol + \int_{\partial M} -\frac{\partial u}{\partial \nu} \cdot v dS_{g} \\
& \geqslant \int_{M} \nabla_{g} u \cdot \nabla_{g} v \dvol + \int_{\partial M} \kappa uv dS_{g}.
\end{align*}
Denote $ w = \max_{M}(u, 0) $. It is immediate that $ w \geqslant 0 $, $ w \in H^{1}(M, g) \cap C^{0}(M) $. Take $ v = w $ in the inequality above, it follows that
\begin{equation*}
0 \geqslant \int_{M} \nabla_{g} u \cdot \nabla_{g} w \dvol + \int_{\partial M} \kappa uw dS_{g} \geqslant \int_{M} \lvert \nabla_{g} w \rvert^{2} \dvol + \int_{\partial M} \kappa w^{2} dS_{g} \geqslant 0.
\end{equation*}
It follows that $ w \equiv 0 $ on $ M $, including the boundary. Therefore $ u \leqslant 0 $.
\end{proof}
\begin{remark}\label{pre:re1}
It is immediate that if there is some $ u $ such that 
\begin{equation}\label{pre:eqn4}
-\Delta_{g} u \geqslant 0, \frac{\partial u}{\partial \nu} + \kappa u \geqslant 0,
\end{equation}
then $ u \geqslant 0 $. It is easily seen by replacing $ u \mapsto -u $ and apply Proposition \ref{pre:prop1}.

In addition, a similar argument applies to the smooth function $ v $ with $ -\Delta_{g} v + Av \leqslant 0 \; {\rm in} \; M, \frac{\partial v}{\partial \nu} \leqslant 0 \; {\rm on} \; \partial M $, which follows that $ v \leqslant 0 $ on $ M $.
\end{remark}
\medskip

We now introduce a general version of the monotone iteration scheme for semi-linear elliptic PDE with leading differential operator $ -\Delta_{g} $, plus a nonlinear oblique boundary condition.
\begin{theorem}\cite[Thm.~2.3]{XU8}\label{pre:thm1}
Let $ (M, \partial M, g) $ be a compact manifold with non-empty smooth boundary, $ n = \dim M \geqslant 2 $. Let $ q > n $ be a positive integer. Let $ F(\cdot, \cdot), G(\cdot, \cdot) : M \times \R \rightarrow \R $ be smooth functions. Let $ \nu $ be the unit outward normal vector along $ \partial M $. Let $ \sigma $ be some nonnegative constant. If

(i) there exists two functions $ u_{+} \in \calC^{\infty}(M) $ and $ u_{-} \in \calC^{0}(M) \cap H^{1}(M, g) $ such that
\begin{equation}\label{pre:eqn5}
\begin{split}
-\Delta_{g} u_{+} & \geqslant F(\cdot, u) \; {\rm in} \; M, \frac{\partial u}{\partial \nu} + \sigma u \geqslant G(\cdot, u_{+}) \; {\rm on} \; \partial M; \\
-\Delta_{g} u_{-} & \leqslant F(\cdot, u) \; {\rm in} \; M, \frac{\partial u}{\partial \nu} + \sigma u \leqslant G(\cdot, u_{-}) \; {\rm on} \; \partial M,
\end{split}
\end{equation}
where the sub-solution may hold in the weak sense; and

(ii) in addition, $ \sup_{M} \lvert G(\cdot, u_{+}) \rvert, \sup_{M} \lvert \nabla G(\cdot, u_{+}) \rvert $ are small enough;

(iii) furthermore, $ u_{+} \geqslant u_{-} $ pointwise on $ M $;

then there exists a smooth function $ u \in \calC^{\infty}(M) $ with $ u_{-} \leqslant u \leqslant u_{+} $ such that
\begin{equation}\label{pre:eqn6}
-\Delta_{g} u = F(\cdot, u) \; {\rm in} \; M, \frac{\partial u}{\partial \nu} + \sigma u = G(\cdot, u) \; {\rm on} \; \partial M.
\end{equation}
\end{theorem}
\medskip

To close this section, we give a $ W^{s, p} $-type elliptic regularity on $ (M, \partial M, g) $ with general dimensions.
\begin{theorem}\cite[Thm.~2.2]{XU5}\label{pre:thm2} 
Let $ (M, \partial M, g) $ be a compact manifold with non-empty smooth boundary. Choose any $ p > \dim M, p \in \mathbb{N} $. Let $ L: \calC^{\infty}(M) \rightarrow \calC^{\infty}(M) $ be a uniform second order elliptic operator on $ M $ with smooth coefficients up to $ \partial M $ and can be extended to $ L : W^{2, p}(M, g) \rightarrow \calL^{p}(M, g) $. Let $ f \in \calL^{p}(M, g), \tilde{f} \in W^{1, p}(M, g) $. Let $ u \in H^{1}(M, g) $ be a weak solution of the following boundary value problem
\begin{equation}\label{pre:eqn7}
L u = f \; {\rm in} \; M, Bu = \frac{\partial u}{\partial \nu} + c(x) u = \tilde{f} \; {\rm on} \; \partial M.
\end{equation}
Here $ c \in \calC^{\infty}(M) $. Assume also that $ \text{Ker}(L) = \lbrace 0 \rbrace $ associated with the homogeneous Robin boundary condition. If, in addition, $ u \in \calL^{p}(M, g) $, then $ u \in W^{2, p}(M, g) $ with the following estimates
\begin{equation}\label{pre:eqn8}
\lVert u \rVert_{W^{2, p}(M, g)} \leqslant C' \left(\lVert Lu \rVert_{\calL^{p}(M, g)} + \lVert Bu \rVert_{W^{1, p}(M, g)} \right)
\end{equation}
Here $ C' $ depends on $ L, p, c $ and the manifold $ (M, \partial M, g) $, and is independent of $ u $.
\end{theorem}

\section{Brezis-Merle type Equations with Oblique Boundary Conditions}
In this section, we consider two Brezis-Merle type equations with oblique boundary conditions on general compact manifolds with non-empty smooth boundary $ (M, \partial M, g) $, $ \dim M = n \geqslant 2 $. Given the functions $ K, \sigma \in \calC^{\infty}(M) $, the first one is
\begin{equation}\label{BM:eqn1}
-\Delta_{g} u = K e^{2u} \; {\rm in} \; M, \frac{\partial u}{\partial \nu} + \kappa u = c\sigma e^{u} \; {\rm on} \; \partial M.
\end{equation}
Here $ \kappa > 0 $ is some fixed constant. and $ c > 0 $ is some constant. 
We will show by using monotone iteration scheme that there exist positive constants $ C_{0}, C_{1} $, depending on $ (M, \partial M, g) $, $ \kappa $, elliptic regularity, and the nonlinear structure on the boundary condition, such that when $ c \leqslant C_{1} $ the PDE (\ref{BM:eqn1}) will always have a smooth solution $ u $ for arbitrary $ \sigma $ and $ K \leqslant C_{0} $.
\medskip

The second one is
\begin{equation}\label{BM:eqn1a}
-\Delta_{g} u + Au = Ke^{2u} \; {\rm in} \; M, \frac{\partial u}{\partial \nu} = c' \sigma e^{u} \; {\rm on} \; \partial M.
\end{equation}
with some fixed constant $ A > 0  $ and some constant $ c' > 0 $. We will show by using monotone iteration scheme that there exist positive constants $ C_{2}, C_{3} $, depending on $ (M, \partial M, g) $, $ A $, elliptic regularity, and the nonlinear structure on the boundary condition, such that when $ c \leqslant C_{3} $ the PDE (\ref{BM:eqn1a}) will always have a smooth solution $ u $ for arbitrary $ \sigma $ and $ K \leqslant C_{2} $.
\medskip

When $ g = g_{e} $, the Euclidean metric, the existence of solutions of the PDE
\begin{equation}\label{BM:eqn2}
-\Delta_{g} u = K e^{2u},
\end{equation}
and the compactness properties are known as Brezis-Merle problems. In \cite{BreMe}, (\ref{BM:eqn2}) with constant coefficient $ K $ and Dirichlet boundary condition within a bounded domain $ \Omega \subset \R^{2} $ was studied; in \cite{LiSha}, a blow-up analysis for sequences of solutions of (\ref{BM:eqn2}) was given; in \cite{ChenLi}, (\ref{BM:eqn2}) on unbounded 2-dimensional Euclidean domain was discussed. Foremost, equations of (\ref{BM:eqn2}) appear in problems of prescribing Gaussian curvatures on closed Riemann surfaces, see \cite{AC, KW2, XU9, XU8} and the references therein. With appropriate boundary conditions, Brezis-Merle type equations appear in the study of prescribing both Gaussian and geodesic curvatures, see \cite{CV, HO}, etc. We point out that (\ref{BM:eqn1}) and (\ref{BM:eqn1a}) are not the right PDE for prescribing Gaussian and geodesic curvature problems, but they play an important role in the Kazdan-Warner problem on compact Riemann surfaces, as we shall see in later sections.

We now give results for the existence of the solution of the first equation (\ref{BM:eqn1}).
\begin{theorem}\label{BM:thm1}
Let $ (M, \partial M, g) $ be a compact manifold with non-empty smooth boundary, $ n = \dim M \geqslant 2 $. Let $ \kappa > 0 $ be a fixed constant. Let $ K, \sigma \in \calC^{\infty}(M) $ be given functions. If $ K \leqslant C_{0} $ on $ M $, and $ c \leqslant C_{1} $ for some positive constants $ C_{0}, C_{1} $, (\ref{BM:eqn1}) has a smooth solution $ u $ with given $ K $ and $ \sigma $.
\end{theorem}
\begin{proof} We construct sub- and super-solutions of (\ref{BM:eqn1}). For super-solution, we consider
\begin{equation*}
-\Delta_{g} u_{0} = f \; {\rm in} \; M, \frac{\partial u_{0}}{\partial \nu} + \kappa u_{0} = a  \; {\rm on} \; \partial M.
\end{equation*}
Here $ f > 0 $ is some smooth function and $ a > 0 $ is some positive constant.By standard elliptic theory, there exists a smooth solution of the equation above. By maximum principle in Proposition \ref{pre:prop1} and Remark \ref{pre:re1}, we conclud that $ u_{0} > 0 $ on $ M $. Clearly, there exists some positive constants $ C_{0}, C_{1}  $ such that
\begin{equation*}
f \geqslant C_{0} e^{2u_{0}} \; {\rm on} \; M, a \geqslant C_{1} \sigma e^{u_{0}} \; {\rm on} \; \partial M.
\end{equation*}
Thus if $ K \leqslant C_{0} $ on $ M $, it follows that
\begin{equation}\label{BM:eqn3}
u_{+} : = u_{0} \Rightarrow -\Delta_{g} u_{+} \geqslant K e^{2u_{+}} \; {\rm on} \; M, \frac{\partial u_{+}}{\partial \nu} + cu_{+} \geqslant c \sigma e^{u_{+}} \; {\rm on} \; \partial M, \forall 0 \leqslant c \leqslant C_{1}.
\end{equation}
Note that both $ C_{0}, C_{1} $ could be very small. For sub-solution, we consider
\begin{equation}\label{BM:eqn4}
-\Delta_{g} u_{1} = b_{1} \; {\rm in} \; M, \frac{\partial u_{1}}{\partial \nu} = b_{2} \; {\rm on} \; \partial M, \int_{M} b_{1} \dvol = -\int_{\partial M} b_{2} dS_{g}, b_{1} < 0, b_{2} > 0.
\end{equation}
Due to the choices of constants $ b_{1} $ and $ b_{2} $, (\ref{BM:eqn4}) has a smooth solution $ u_{1} $. Choose $ C \gg 1 $ such that
\begin{equation*}
u_{2} : = u_{1} - C < u_{+} \; {\rm on} \; M \; {\rm pointwise}.
\end{equation*}
Clearly $ u_{2} \in \calC^{\infty}(M) $ solves (\ref{BM:eqn4}). In addition, we take $ C $ to be large enough such that
\begin{align*}
b_{1} & \leqslant K e^{2u_{2}} = K e^{2u_{1}} \cdot e^{-2C} \; {\rm on} \; M, \\
b_{2} + \kappa u_{2} & = b_{2} + \kappa u_{1} - \kappa C \leqslant C_{1} \sigma e^{u_{2}} = C_{1} \sigma e^{u_{1}} \cdot e^{-C} \leqslant c \sigma e^{u_{1}} \cdot e^{-C} \; {\rm on} \; \partial M, \forall 0 \leqslant c \leqslant C_{1}.
\end{align*}
These can be done since $ b_{1} < 0 $, $ b_{2} + \kappa u_{1} - \kappa C < 0 $ when $ C $ large, and $ e^{-C} $ is very small when $ C $ large enough. It follows that
\begin{equation}\label{BM:eqn5}
u_{-} : = u_{2} \Rightarrow -\Delta_{g} u_{-} \leqslant K e^{2u_{-}} \; {\rm on} \; M, \frac{\partial u_{-}}{\partial \nu} + \kappa u_{-} \leqslant c\sigma e^{u_{-}} \; {\rm on} \; \partial M.
\end{equation}
We may need to make the constant $ C_{1} $ even smaller if necessary so that the restrictions (ii) in Theorem \ref{pre:thm1} hold. Note that both (\ref{BM:eqn3}) and (\ref{BM:eqn5}) holds for smaller threshold $ C_{1} $. Since $ u_{+} \geqslant u_{-} $ pointwise by the setup, we conclude by Theorem \ref{pre:thm1} that (\ref{BM:eqn1}) has a smooth solution $ u $ with given $ K $ and $ \sigma $, provided that $ K \leqslant C_{0} $ and $ c \leqslant C_{1} $ for some positive constant $ C_{0} $ and $ C_{1} $.
\end{proof}
\begin{remark}\label{BM:re1}
We are not giving the optimal $ C_{0} $ and $ C_{1} $ in the previous proof. The upper bound $ C_{0} $ is due to a pure analytic reason. It is interesting to see that similar upper bounds are imposed in the results of Brezis-Merle problems, see e.g. \cite{BreMe, LiSha}.
\end{remark}
\medskip

We now turn to the second equation (\ref{BM:eqn1a}). The strategy of the proof is quite similar as in Theorem \ref{BM:thm1}.
\begin{theorem}\label{BM:thm2}
Let $ (M, \partial M, g) $ be a compact manifold with non-empty smooth boundary, $ n = \dim M \geqslant 2 $. Let $ A > 0 $ be a fixed constant. Let $ K, \sigma \in \calC^{\infty}(M) $ be given functions. If $ K \leqslant C_{2} $ on $ M $, and $ c \leqslant C_{3} $ for some positive constants $ C_{2}, C_{3} $, (\ref{BM:eqn1a}) has a smooth solution $ u $ with given $ K $ and $ \sigma $.
\end{theorem}
\begin{proof}
We construct sub- and super-solutions. For super-solution, we consider
\begin{equation}\label{BM:eqn6}
-\Delta_{g} u_{0} + Au_{0} = a \: {\rm in} \; M, \frac{\partial u_{0}}{\partial \nu} = b \; {\rm on} \; \partial M.
\end{equation}
Here $ a, b > 0 $ are some constant in (\ref{BM:eqn6}). By standard elliptic theory, (\ref{BM:eqn6}) has a smooth solution $ u_{0} $. Since $a, b > 0 $, maximum principle implies that $ u_{0} \geqslant 0 $ on $ M $, up to the boundary. Thus if $ K \leqslant C_{2} = a \cdot e^{-2\max_{M} u_{0}} $, it follows from (\ref{BM:eqn6}) that
\begin{equation}\label{BM:eqn7}
u_{+} : = u_{0} \Rightarrow -\Delta_{g} u_{+} + A u_{+} = a \geqslant K e^{2u_{+}} \; {\rm on} \; M, \frac{\partial u_{+}}{\partial \nu} = b \geqslant c' \sigma e^{u_{+}} \; {\rm on} \; \partial M
\end{equation}
for every $ c' \in [0, C_{3}] $. The constant $ C_{3} $ is determined by $ b, \sigma $ and $ u_{+} $.

For sub-solution, we consider
\begin{equation}\label{BM:eqn8}
-\Delta_{g} u_{1} = b_{1} \; {\rm in} \; M, \frac{\partial u_{1}}{\partial \nu} = b_{2} \; {\rm on} \; \partial M, \int_{M} b_{1} \dvol = -\int_{\partial M} b_{2} dS_{g}, b_{1} > 0, b_{2} < 0.
\end{equation}
Due to the choices of constants $ b_{1} $ and $ b_{2} $, (\ref{BM:eqn4}) has a smooth solution $ u_{1} $. Choose $ C \gg 1 $ such that
\begin{equation*}
u_{2} : = u_{1} - C, u_{2} \leqslant u_{+} \; {\rm on} \; M.
\end{equation*}
Clearly $ u_{2} $ solves (\ref{BM:eqn6}), we can take $ C $ to be even larger so that
\begin{equation}\label{BM:eqn9}
\begin{split}
u_{-} : = u_{2} \Rightarrow & -\Delta_{g} u_{-} + Au_{-} = b_{1} + A(u_{1} - C) < Ke^{2u_{1}} \cdot e^{-2C} = K e^{2u_{-}} \; {\rm on} \; M; \\
& \frac{\partial u_{-}}{\partial \nu} = b_{2} \leqslant c' \sigma e^{u_{1}} \cdot e^{-C} = c' \sigma e^{u_{-}} \; {\rm on} \; \partial M.
\end{split}
\end{equation}
Note that the second inequality in (\ref{BM:eqn9}) holds for every $ c' $ chosen in (\ref{BM:eqn7}), with an appropriate choice of $ C $.

Lastly we take the threshold $ C_{3} $ to be small enough so that the condition (ii) of Theorem \ref{pre:thm1} is satisfied. Applying the monotone iteration scheme, we conclude that there exists some $ u \in \calC^{\infty}(M) $ that solves (\ref{BM:eqn1a}) for $ K \leqslant C_{2} $ and arbitrary $ \sigma $ with some $ c' \leqslant C_{3} $ .
\end{proof}
\begin{remark}\label{BM:re2}
Analogous to Remark \ref{BM:re1}, the choices of $ C_{2}, C_{3} $ are obviously not optimal.
\end{remark}

\section{Prescribing Arbitrary Gaussian or Geodesic Curvatures}
Recall the Kazdan-Warner problem on closed Riemann surface, clearly not every smooth function can be realized as a prescribed Gaussian curvature. One of the well-known obstructions is the Bourguignon-Ezin type obstruction \cite{BE}, which states that any prescribed Gaussian curvature $ K $ must satisfy
\begin{equation*}
\int_{M_{0}} \nabla_{g} K \cdot X_{g} \dvol = 0.
\end{equation*}
Here $ (M_{0}, g) $ is any closed manifold of dimension at least $ 2 $, and $ X_{g} $ is the associated conformal killing field. 

The results are surprisingly different on compact manifolds with non-empty smooth boundary. We are able to show that any given function $ K \in \calC^{\infty}(M) $ can be realized as a Gaussian curvature function for some metric $ \tilde{g} \in [g] $, associated with some geodesic curvature $ \sigma_{\tilde{g}} $ that is determined by $ K $, $ g $ and the topology of $ M $. If we fix the boundary condition, e.g. the minimal boundary, then there are obstructions on the choices of prescribing functions, due to Gauss-Bonnet.

Similarly, we can also show that any given function $ \sigma \in \calC^{\infty}(\partial M) $ can be realized as a geodesic curvature function for some metric pointwise conformal to $ g $, associated with some Gaussian curvature determined by $ M $ and $ \sigma $. Although we have less restrictions on extrinsic curvatures, Escobar \cite{ESC3} showed that the necessary condition of prescribing mean curvature function $ h $ on $ \partial \mathbb{B}^{n} $ with scalar-flat metric in $ \mathbb{B}^{n} $ is
\begin{equation*}
\int_{\partial \mathbb{B}^{n}} Y_{g} \cdot \nabla h dS_{g} = 0.
\end{equation*}
Here $ Y_{g} $ are conformal killing field on $ \partial \mathbb{B}^{n} $.

Our surprising results are due to the Brezis-Merle type equations introduced in \S3. Let's consider the prescribing Gaussian curvature within a conformal class $ [g] $ first.
\begin{theorem}\label{ARB:thm1}
Let $ (M, \partial M, g) $ be a compact Riemann surface with non-empty smooth boundary. For any given function $ K \in \calC^{\infty}(M) $, there exists a conformal metric $ \tilde{g} \in [g] $ such that $ K_{\tilde{g}} = K $, associated with some geodesic curvature $ \sigma_{\tilde{g}} $ determined by $ K $ and $ (M, \partial M, g) $.
\end{theorem}
\begin{proof}
We scale $ K $ to $ bK $ with some positive constant $ b \ll 1 $, if necessary, so that $ \sup_{M} bK \leqslant C_{0} $ with the constant $ C_{0} $ given in Theorem \ref{BM:thm1}. Applying Theorem \ref{BM:thm1} with $ bK $ and $ \sigma = 0 $, it follows that
\begin{equation}\label{ARB:eqn1}
-\Delta_{g} u = bK e^{2u} \; {\rm in} \; M, \frac{\partial u}{\partial \nu} + \kappa u = 0 \; {\rm on} \; \partial M
\end{equation}
admits a smooth solution. When $ \chi(M) < 0 $, we may assume that the model case is $ K_{g} = 0 $ and $ \sigma_{g} = -1 $ due to Theorem \ref{intro:thm1}. It follows that the same solution $ u $ solves the following equation
\begin{equation*}
-\Delta_{g} u = bK e^{2u} \; {\rm in} \; M, \frac{\partial u}{\partial \nu} - 1 = \frac{-\kappa u - 1}{e^{u}} e^{u} : = \sigma_{1} e^{u} \; {\rm on} \; \partial M.
\end{equation*}
After a scaling, we conclude that the metric $ \tilde{g}_{1} = b e^{2u} g $ has $ K_{\tilde{g}_{1}} = K $ and $ \sigma_{\tilde{g}_{1}} = b^{-\frac{1}{2}} \sigma_{1} $.

When $ \chi(M) = 0 $, the model case is $ K_{g} = \sigma_{g} = 0 $, therefore the solution $ u $ also solves
\begin{equation*}
-\Delta_{g} u = bK e^{2u} \; {\rm in} \; M, \frac{\partial u}{\partial \nu} = \frac{-\kappa u}{e^{u}} e^{u} : = \sigma_{2} e^{u} \; {\rm on} \; \partial M,
\end{equation*}
which follows that $ K_{\tilde{g}_{2}} = K $ and $ \sigma_{\tilde{g}_{2}} = b^{-\frac{1}{2}} \sigma_{2} $ for $ \tilde{g}_{2} = be^{2u} g $.

Lastly when $ \chi(M) > 0 $, the model case is $ K_{g} = 0 $ and $ \sigma_{g} = 1 $ by Theorem \ref{intro:thm1}. The function $ u $ given in (\ref{ARB:eqn1} also solves
\begin{equation*}
-\Delta_{g} u = bK e^{2u} \; {\rm in} \; M, \frac{\partial u}{\partial \nu} + 1 = \frac{-\kappa u + 1}{e^{u}} e^{u} : = \sigma_{3} e^{u} \; {\rm on} \; \partial M.
\end{equation*}
Hence $ \tilde{g}_{3} = be^{2u} g $ has Gaussian curvature $ K $ and geodesic curvature $ b^{-\frac{1}{2}} \sigma_{3} $.
\end{proof}
\medskip

Due to a very similar argument, we can show that any function can be realized as a geodesic curvature for some conformal metric.
\begin{theorem}\label{ARB:thm2}
Let $ (M, \partial M, g) $ be a compact Riemann surface with non-empty smooth boundary. For any given function $ \sigma \in \calC^{\infty}(M) $, there exists a conformal metric $ \tilde{g} \in [g] $ such that $ \sigma_{\tilde{g}} = \sigma $, associated with some Gaussian curvature $ K_{\tilde{g}} $ determined by $ \sigma $ and $ (M, \partial M, g) $.
\end{theorem}
\begin{proof}
This time we apply $ K = 0 $ and the given $ \sigma $ in Theorem \ref{BM:thm2}. Hence there exist some small enough constant $ c $ such that
\begin{equation}\label{ARB:eqn2}
-\Delta_{g} u + Au = 0 \; {\rm in} \; M, \frac{\partial u}{\partial \nu} = c\sigma e^{u} \; {\rm on} \; \partial M
\end{equation}
admit a smooth solution $ u $.

When $ \chi(M) < 0 $, we may assume that the model case is now $ K_{g} = - 1 $ and $ \sigma_{g} = 0 $ by the classical uniformization theorem. Thus the same function $ u $ solves
\begin{equation*}
-\Delta_{g} u - 1 = \frac{-Au - 1}{e^{2u}} e^{2u} : = K_{1} e^{2u} \; {\rm in} \; M, \frac{\partial u}{\partial \nu} = c\sigma e^{u} \; {\rm on} \; \partial M.
\end{equation*}
Hence the metric $ \tilde{g}_{1} = c^{2} e^{2u} g $ has the desired geodesic curvature $ \sigma_{\tilde{g}_{1}} = \sigma $ and $ K_{\tilde{g}_{1}} = c^{-2} K_{1} $.

When $ \chi(M) = 0 $ and $ \chi(M) > 0 $, respectively, we see that the same equation $ u $ solves the following two equations
\begin{align*}
-\Delta_{g} u & = \frac{-Au}{e^{2u}} e^{2u} : = K_{2} e^{2u} \; {\rm in} \; M, \frac{\partial u}{\partial \nu} = c\sigma e^{u} \; {\rm on} \; \partial M; \\
-\Delta_{g} u + 1 & = \frac{-Au + 1}{e^{2u}} e^{2u} : = K_{3} e^{2u} \: {\rm in} \; M, \frac{\partial u}{\partial \nu} = c\sigma e^{u} \; {\rm on} \; \partial M.
\end{align*}
Thus the conclusions follow clearly.
\end{proof}

\section{Prescribing Gaussian and Geodesic Curvatures with $ \chi(M) = 0 $}.
Although we have the interesting results in \S4, we are willing to choose the desired Gaussian and geodesic curvatures simultaneously. We can do a little bit better when $ \chi(M) = 0 $.

In \cite{CV, XU9}, results for prescribing Gaussian curvatures with minimal boundary on $ (M, \partial M, g) $, $ \chi(M) = 0 $, for some conformal metric in $ [g] $ were given: any function $ K $ that can be realized as Gaussian curvature of some metric $ \tilde{g} \in [g] $ with minimal boundary if and only if either (i) $ K \equiv 0 $ or (ii) $ \int_{M} K \dvol < 0 $ and $ K $ changes sign.

In this section, we introduce results for prescribing both nontrivial Gaussian and geodesic curvature functions $ K, \sigma \in \calC^{\infty}(M) $ for some pointwise conformal metric on $ (M, \partial M, g) $, provided that $ \chi(M) = 0 $. The model case is $ K_{g} = \sigma_{g} = 0 $, by the uniformization theorem. In particular, we will consider the following cases:
\begin{equation}\label{zero:eqn0}
 K < 0 \; \text{ everywhere on } \;  M, \sigma > 0 \; \text{everywhere on} \; \partial M.
\end{equation}
Based on the model case, the problem is reduced to the existence of smooth solutions of the equation
\begin{equation}\label{zero:eqn1}
-\Delta_{g} u = K e^{2u} \; {\rm in} \; M, \frac{\partial u}{\partial \nu} = c \sigma e^{u} \; {\rm on} \; \partial M.
\end{equation}
Here we intentionally introduce the constant $ c > 0 $ on the boundary conditions for computational and notational convenience. Recall that (\ref{zero:eqn1}) says that the metric $ \tilde{g} = e^{2u} g $ has $ K_{\tilde{g}} = K $ and $ \sigma_{\tilde{g}} = c\sigma $ by assuming the solvability. Again we try to construct sub- and super-solutions of (\ref{zero:eqn1}). The key step is the solvability of (\ref{BM:eqn1}), the Brezis-Merle type equation that is discussed in \S3.

We need a small lemma to convert the PDE (\ref{zero:eqn1}) into another equivalent version:
\begin{lemma}\cite[Lemma.~5.1]{XU8}\label{zero:lemma1}
Let $ (M, \partial M, g) $ be a compact Riemann surface with non-empty smooth boundary. Let $ K, \sigma \in \calC^{\infty}(M) $ be given functions. Then there exists some function $ u \in \calC^{\infty}(M) $ satisfying
\begin{equation}\label{zero:eqn2}
-\Delta_{g} u + K_{g} \geqslant K e^{2u} \; {\rm in} \; M, \frac{\partial u}{\partial \nu} + \sigma_{g} \geqslant \sigma e^{u} \; {\rm on} \; \partial M
\end{equation}
if and only if there exists some positive function $ w \in \calC^{\infty}(M) $ satisfying
\begin{equation}\label{zero:eqn3}
-\Delta_{g} w - 2wK_{g} + \frac{\lvert \nabla_{g} w \rvert^{2}}{w} \leqslant -2K \; {\rm in} \; M, \frac{\partial w}{\partial \nu} - 2w \sigma_{g} \leqslant -2\sigma w^{\frac{1}{2}} \; {\rm on} \; \partial M.
\end{equation}
In addition, the equalities in (\ref{zero:eqn2}) hold if and only if the equalities in (\ref{zero:eqn3}) hold; Furthermore, the inequalities in (\ref{zero:eqn2}) are in the reverse direction if and only if the inequalities in (\ref{zero:eqn3}) are in the reverse direction, correspondingly.
\end{lemma}
\begin{remark}\label{zero:re1}
When $ K_{g} $ and $ \sigma_{g} $ are Gaussian and geodesic curvatures, respectively, the formula (\ref{zero:eqn2}) with equality has its geometric meaning. But Lemma \ref{zero:lemma1} still holds for arbitrary smooth functions $ K_{g} $ and $ \sigma_{g} $. The relation between $ u $ in (\ref{zero:eqn2}) and $ w $ in (\ref{zero:eqn3}) is $ w = e^{-2u} $.
\end{remark}

\begin{theorem}\label{zero:thm1}
Let $ (M, \partial M, g) $ be a compact Riemann surface with non-empty smooth boundary, $ \chi(M) = 0 $. Let $ K, \sigma \in \calC^{\infty}(M) $ be given functions that satisfying (\ref{zero:eqn0}). Then there exists a positive constant $ D $ such that (\ref{zero:eqn1}) has a smooth solution for given $ K, \sigma $ and all $ c \in (0, D] $.
\end{theorem}
\begin{proof}
We construct global sub- and super-solutions. For super-solution, we consider the PDE of type (\ref{BM:eqn1}):
\begin{equation}\label{zero:eqn4}
-\Delta_{g} u_{0} = K e^{2u_{0}} \; {\rm in} \; M, \frac{\partial u_{0}}{\partial \nu} + \kappa u = 0.
\end{equation}
By Theorem \ref{BM:thm1}, the equation (\ref{zero:eqn4}) admits a smooth solution $ u_{0} $ as $ K < 0 < C_{0} $ for $ C_{0} $ given in Theorem \ref{BM:thm1}. Since $ K < 0 $ everywhere on $ M $, the maximum principle in Proposition \ref{pre:prop1} implies that $ u_{0} \leqslant 0 $ on $ M $. Clearly $ u_{0} \not\equiv 0 $. It is immediate that
\begin{equation}\label{zero:eqn5}
u_{+} : = u_{0} \Rightarrow -\Delta_{g} u_{+} \geqslant K e^{2u_{+}} \; {\rm on} \; M, \frac{\partial u_{+}}{\partial \nu} = -\kappa u_{+} \geqslant c\sigma e^{u_{+}} \; {\rm on} \; \partial M.
\end{equation}
Note that the inequality of the boundary condition holds when the positive constant $ c \ll 1 $ is small enough, since $ -\kappa u_{+} > 0 $. For the sub-solution, we consider the PDE
\begin{equation*}
-a\Delta_{g} w_{0} = -2K \; {\rm in} \; M, \frac{\partial w_{0}}{\partial \nu} = A \; {\rm on} \; \partial M, A < 0, \int_{M} -2K \dvol = -\int_{\partial M} A dS_{g}.
\end{equation*}
Note that the relation between $ u $ in (\ref{zero:eqn2}) and $ w $ in (\ref{zero:eqn3}) is $ w = e^{-2u} $. It follows that a very large $ w > 1 $ implies a very negative $ u < 0 $. Choose a large enough $ C \gg 1 $ such that 
\begin{equation}\label{zero:eqn6}
w_{1} : = w + C \Rightarrow -\frac{1}{2} \log w_{1} \leqslant u_{+} \; {\rm on} \; M, A \geqslant -2c\sigma w_{1}^{\frac{1}{2}} \; {\rm on} \; \partial M.
\end{equation}
Note that when $ c $ is smaller, we can always choose $ C $ to be larger, and the middle inequality in (\ref{zero:eqn6}) still holds. In addition, when $ C $ is larger, $ e^{-\frac{1}{2} \log w_{1}} $ is even smaller, so a larger choice of $ C $ will not affect the condition (ii) in Theorem \ref{pre:thm1} for the monotone iteration scheme. With fixed $ u_{+} $, we choose $ c $ small enough so that the condition (ii) in Theorem \ref{pre:thm1} holds, then choose $ C $ large enough such that (\ref{zero:eqn6}) holds. By (\ref{zero:eqn6}),
\begin{equation*}
-\Delta_{g} w_{1} + \frac{\lvert \nabla_{g} w_{1} \rvert^{2}}{w_{1}} \geqslant -2K \; {\rm on} \; M, \frac{\partial w_{1}}{\partial \nu} = A \geqslant -2c\sigma w_{1}^{\frac{1}{2}} \; {\rm on} \; \partial M.
\end{equation*}
Applying Lemma \ref{zero:lemma1},
\begin{equation}\label{zero:eqn7}
u_{-} : = -\frac{1}{2} \log w_{1} \Rightarrow -\Delta_{g} u_{-} \leqslant K e^{u_{-}} \; {\rm on} \; M, \frac{\partial u_{-}}{\partial \nu} \leqslant c\sigma e^{u_{-}} \; {\rm on} \; \partial M.
\end{equation}
Clearly (\ref{zero:eqn6}) holds for an even smaller positive constant $ c $ with a larger $ C $ in (\ref{zero:eqn6}), and hence a more negative $ u_{-} $. The upper bound $ D $ for the constant $ c $ thus can be determined by the largest possible $ c $ such that (\ref{zero:eqn5}), (\ref{zero:eqn6}) and (\ref{zero:eqn7}) are satisfied. By (\ref{zero:eqn6}), we have $ u_{+} \geqslant u_{-} $ pointwise.  Applying Theorem \ref{pre:thm1} for (\ref{zero:eqn5}) and (\ref{zero:eqn7}), there exists a smooth function $ u $ which solves (\ref{zero:eqn1}) for small enough $ c $, $ \sigma $ and $ K $. Therefore the metric $ \tilde{g}= e^{2u} g $ has $ K_{\tilde{g}} = K $ and $ \sigma_{\tilde{g}} = c\sigma $. 
\end{proof}
\medskip

\section{Prescribing Gaussian and Geodesic Curvatures with $ \chi(M) < 0 $}
If $ \chi(M) < 0 $, we have shown in \cite{XU9} that any $ K < 0 $ and $ c\sigma $ with small enough $ c \in (0, 1) $ and arbitrary $ \sigma $ can be realized as Gaussian and geodesic curvature functions with $ \chi(M) < 0 $, respectively; in addition, there is a restriction for prescribing functions that are positive somewhere. In this section, we consider the following case,
\begin{equation}\label{neg:eqn1}
K \; \text{changes sign and positive somewhere on} \; M, \sigma \geqslant 0 \; \text{everywhere on} \; \partial M.
\end{equation}
Recall that the smooth functions $ K $ and $ \sigma $ can be realized as Gaussian and geodesic curvatures for some conformal metric if
\begin{equation}\label{neg:eqn2}
-\Delta_{g} u + K_{g} = Ke^{2u} \; {\rm in} \; M, \frac{\partial u}{\partial \nu} + \sigma_{g} u = \sigma e^{u} \; {\rm on} \; \partial M.
\end{equation}
Analogous to the necessary condition given in \cite{KW2}, we give necessary conditions for prescribing functions of type (\ref{neg:eqn1}).
\begin{proposition}\label{neg:prop1}
Let $ (M, \partial M, g) $ be a compact Riemann surface with non-empty smooth boundary, $ \chi(M) < 0 $. If (\ref{neg:eqn2}) admits a smooth solution $ u $ for some $ K, \sigma $ that satisfying (\ref{neg:eqn1}), then
\begin{enumerate}[(i)]
\item The unique solution of
\begin{equation}\label{neg:eqn3}
-\Delta_{g} w = - 2K \; {\rm in} \; M, \frac{\partial w}{\partial \nu} + 2w = 0 \; {\rm on} \; \partial M
\end{equation}
is positive on $ M $; 
\item $ \int_{M} K \dvol < 0 $.
\end{enumerate}
\end{proposition} 
\begin{proof} Without loss of generality, we may assume that $ K_{g} = 0 $ and $ \sigma_{g} = - 1 $ due to conformal invariance and Theorem \ref{intro:thm1}. Due to the hypotheses, Lemma \ref{zero:lemma1} indicates that there exists a positive, smooth solution $ w_{0} $ that satisfies
\begin{equation}\label{neg:eqn4}
-\Delta_{g} w_{0} + \frac{\lvert \nabla w_{0} \rvert^{2}}{w_{0}} = -2K \; {\rm in} \; M, \frac{\partial w_{0}}{\partial \nu} + 2w_{0} = -2\sigma w_{0}^{\frac{1}{2}} \; {\rm on} \; \partial M.
\end{equation}
By Proposition \ref{pre:prop1}, we know that the operator in (\ref{neg:eqn3}) is injective. The operator is also self-adjoint, hence by Fr\'edholm alternative, it is invertible. Thus (\ref{neg:eqn3}) admits a unique smooth solution $ w $ for each fixed $ K, \sigma $, the regularity of $ w $ is given by Theorem \ref{pre:thm2}.

Subtract (\ref{neg:eqn4}) by (\ref{neg:eqn3}), it follows that
\begin{equation*}
-\Delta_{g} (w - w_{0}) = \frac{\lvert \nabla w_{0} \rvert^{2}}{w_{0}} \geqslant 0 \; {\rm in} \; M, \frac{\partial (w - w_{0})}{\partial \nu} + 2(w - w_{0}) = 2\sigma w_{0}^{\frac{1}{2}} \geqslant 0 \; {\rm on} \; \partial M.
\end{equation*}
The last inequality above is due to the nonnegativity of $ \sigma $. Apply Proposition \ref{pre:prop1}, it follows that $ w - w_{0} \geqslant 0 $. Hence $ w \geqslant w_{0} > 0 $ on $ M $. It proves (i).

Applying the result of (i), we observe that
\begin{equation*}
\int_{M} - 2K \dvol = - \int_{\partial M} -2w dS_{g} > 0 \Rightarrow \int_{M} K \dvol < 0.
\end{equation*}
\end{proof}
\medskip

We can apply the monotone iteration scheme to consider the situation of prescribing constant geodesic curvature with arbitrary Gaussian curvature function, as the following proposition shows.
\begin{lemma}\label{neg:lemma2}
Let $ (M, \partial M, g) $ be a compact Riemann surface with non-empty smooth boundary. Let $ K < 0 $ everywhere be a given smooth function. The following equation
\begin{equation}\label{neg:eqna1}
-\Delta_{g} u = Ke^{2u} \; {\rm in} \; M, \frac{\partial u}{\partial \nu} - \tilde{\kappa} = 0 \; {\rm on} \; \partial M
\end{equation}
has a smooth solution with small enough positive constant $ \tilde{\kappa} > 0 $.
\end{lemma}
\begin{proof}
By Equation (23) of \cite[Thm.~2.4]{XU8}, if we choose some $ \tilde{\kappa} $ small enough such that $ \tilde{\kappa} \cdot \text{Vol}_{g}(M)^{\frac{1}{q}} < 1 $ for some fixed number $ q > 2 $, then the condition (ii) of Theorem \ref{pre:thm1} holds, since for this very special Neumann boundary condition, the condition (ii) is independent of the choices of the sub- and super-solutions, neither the elliptic regularity. Fix this $ \tilde{\kappa} $.

For super-solution, we consider the PDE
\begin{equation}\label{neg:eqn5}
-\Delta_{g} w_{0} = -K \; {\rm in} \; M, \frac{\partial w_{0}}{\partial \nu} + 2\tilde{\kappa} w_{0} = 0.
\end{equation}
There exists a smooth solution $ w_{0} $ of (\ref{neg:eqn5}) due to standard elliptic theory. By maximum principle in Proposition \ref{pre:prop1}, $ w_{0} \geqslant 0 $ on $ M $ since $ -K > 0 $. By standard maximum principle of second order elliptic operator, we conclude that $ w_{0} > 0 $ in the interior $ M $, thus by the Hopf Lemma, $ w_{0} > 0 $ on the whole manifold, up to the boundary. We then take a small constant $ 0 < \xi \ll 1 $ such that
\begin{equation*}
\frac{\lvert \nabla_{g} (\xi w_{0}) \rvert^{2}}{\xi w_{0}} - \xi K \leqslant - 2K \; {\rm on} \; M.
\end{equation*}
This can be done since both sides are positive. It follows that
\begin{align*}
w : = \xi w_{0} \Rightarrow & -\Delta_{g} w + \frac{\lvert \nabla_{g} w \rvert^{2}}{w} = - \xi K + \frac{\lvert \nabla_{g} (\xi w_{0}) \rvert^{2}}{\xi w_{0}} \leqslant -2K \; {\rm on} \; M, \\
& \frac{\partial w}{\partial \nu} + 2 \tilde{\kappa} w = \xi \cdot \left( \frac{\partial w_{0}}{\partial \nu} + 2\tilde{\kappa} w_{0} \right) = 0.
\end{align*}
Therefore
\begin{equation}\label{neg:eqn6}
u_{+} : = -\frac{1}{2} \log w \Rightarrow -\Delta_{g} u_{+} \geqslant K e^{2u_{+}} \; {\rm in} \; M, \frac{\partial u_{+}}{\partial \nu} - \tilde{\kappa} \geqslant 0 \; {\rm on} \; \partial M.
\end{equation}
For sub-solution, we consider the PDE
\begin{equation*}
-\Delta_{g} u_{0} = A \; {\rm in} \; M, \frac{\partial u_{0}}{\partial \nu} = \tilde{\kappa} \; {\rm on} \; \partial M, A < 0, \int_{M} A \dvol = -\int_{\partial M} \tilde{\kappa} dS_{g}.
\end{equation*}
By standard elliptic theory \cite[Ch.~5]{T}, there exists some smooth $ u_{0} $ that solves the equation right above. Therefore for large enough $ \tilde{C} \gg 1 $, we have
\begin{equation}\label{neg:eqn7}
\begin{split}
u_{-} : = u_{0} - \tilde{C} \Rightarrow & -\Delta_{g} u_{-} = -\Delta_{g} u_{0} = A \leqslant K e^{u_{0}} \cdot e^{-\tilde{C}} = K e^{u_{-}} \; {\rm on} \; M, \\
& \frac{\partial u_{-}}{\partial \nu} - \tilde{\kappa} = \frac{\partial u_{0}}{\partial \nu} - \tilde{\kappa}  = 0 \; {\rm on} \; \partial M, \\
& u_{-} \leqslant u_{+} \; {\rm on} \; M \; {\rm pointwise}.
\end{split}
\end{equation}
The first inequality in (\ref{neg:eqn7}) holds since both $ A $ and $ K e^{u_{-}} $ are negative terms. By (\ref{neg:eqn6}) and (\ref{neg:eqn7}), we conclude from Theorem \ref{pre:thm1} that the equation (\ref{neg:eqna1}) has a smooth solution.
\end{proof}
\begin{remark}\label{neg:re2}
The equation (\ref{neg:eqna1}) is another Brezis-Merle type equation, with nontrivial Neumann boundary condition, which is a special case of (\ref{BM:eqn1}) with $ \kappa = 0 $. The restriction is that we must assume $ K < 0 $ everywhere for (\ref{neg:eqna1}).
\end{remark}
\medskip

\section{Prescribing Gaussian and Geodesic Curvatures with $ \chi(M) > 0 $}
A typical example of a compact Riemann surface $ (M, \partial M, g) $ with $ \chi(M) > 0 $ is the upper-hemisphere. It is analogous to the Nirenberg problem \cite{AC} which considers the prescribing Gaussian curvature problem on $ \mathbb{S}^{2} $, which is a challenging problem since we expect obstructions to some extent due to the topology of the sphere. 

Some people have studied the $ \chi(M) > 0 $ case. One remarkable result is an extension of the classical uniformization theorem, given by \cite{OsPhSa}. Other known results are either for prescribing Gaussian curvature with minimal boundary \cite{ESC3}, or prescribing geodesic curvature with scalar-flat metric \cite{CV}. 

Recall the PDE for prescribing Gaussian and geodesic curvature functions $ K, \sigma \in \calC^{\infty}(M) $:
\begin{equation}\label{ps:eqn1}
-\Delta_{g} u + K_{g} = K e^{2u} \; {\rm in} \; M, \frac{\partial u}{\partial \nu} + \sigma_{g} = c\sigma e^{u} \; {\rm on} \; \partial M, c > 0.
\end{equation}
When neither $ K $ nor $ \sigma $ vanishes identically, the only impossible case is $ K \leqslant 0 $ and $ \sigma \leqslant 0 $, due to Gauss-Bonnet Theorem. In this section, we discuss the possibilities of various choices of prescribing Gaussian and geodesic curvature functions. The results in \S3 seems to be the optimal result without using variational methods or Morse theory. The essential difficulty for this type of method is due to the topological obstruction. Analytically, we mainly use
\begin{equation*}
-\Delta_{g} u = A K_{g} \; {\rm in} \; M, \frac{\partial u}{\partial \nu} = B \sigma_{g} \; {\rm on} \; \partial M
\end{equation*}
to construct one of the barriers. But the order of constants $ A, B $ are determined by the Gauss-Bonnet formula. For example, if $ K_{g} < 0, \sigma_{g} > 0 $ and $ \chi(M) > 0 $, we must have $ A > B $, since otherwise the compatibility condition $ \int_{M} AK_{g} \dvol = -\int_{\partial M} B \sigma_{g} dS_{g} $ cannot hold.
The following proposition shows that there exists at least one example for possible case.
\begin{proposition}\label{ps:prop1}
Let $ (M, \partial M, g) $ be a compact Riemann surface with non-empty smooth boundary, $ \chi(M) > 0 $. There exist smooth functions $ K_{i}, \sigma_{i}, i = 1, \dotso, 5 $, $ K_{i}, \sigma_{i} \not\equiv 0 $ with
\begin{align*}
& K_{1} \geqslant 0 \; \text{everywhere on} \; M, \sigma_{1} \geqslant 0 \; \text{everywhere on} \; M; \\
& K_{2} \; \text{changes sign on} \; M, \sigma_{2} \geqslant 0 \; \text{everywhere on} \; M; \\
& K_{3} \leqslant 0 \; \text{everywhere on} \; M, \sigma_{3} \geqslant 0 \; \text{everywhere on} \; M; \\
& K_{4} \geqslant 0 \; \text{everywhere on} \; M, \sigma_{4} \; \text{changes sign on} \; M; \\
& K_{5} \geqslant 0 \; \text{everywhere on} \; M, \sigma_{5} \leqslant 0 \; \text{everywhere on} \; M; \\
& K_{6} \; \text{changes sign on} \; M, \sigma_{6} \; \text{changes sign on} \; M; \\
& K_{7} \leqslant 0 \; \text{everywhere on} \; M, \sigma_{7} \; \text{changes sign on} \; M; \\
& K_{8} \; \text{changes sign on} \; M, \sigma_{8} \leqslant 0 \; \text{everywhere on} \; M; \\
\end{align*}
such that $ K_{i}, \sigma_{i} $ can be realized as Gaussian and geodesic curvatures of some conformal metric $ g_{i} \in [g] $, $ i = 1, \dotso, 8 $ respectively.
\end{proposition}
\begin{proof}
We construct $ K_{i} $ and $ \sigma_{i} $. For the first three cases, we choose the model case $ K_{g} = 0, \sigma_{g} = 1 $, due to Theorem \ref{intro:thm1}. Consider the following equations
\begin{equation}\label{ps:eqn2}
-\Delta_{g} u_{i} = f_{i} \; {\rm in} \; M, \frac{\partial u_{i}}{\partial \nu} = A_{i} \; {\rm on} \; \partial M, \int_{M} f_{i} \dvol = -\int_{\partial M} A_{i} dS_{g}, i = 1, 2, 3
\end{equation}
We choose $ f_{1} \geqslant 0 $ on $ M $, $ f_{2} $ changing sign on $ M $ and $ f_{3} \leqslant 0 $ on $ M $. Furthermore, we choose $ f_{i}, i = 1, 2, 3 $ such that $ \lvert A_{i} \rvert < 1 $. By standard elliptic theory, (\ref{ps:eqn2}) can be solved by smooth functions $ u_{1}, u_{2}, u_{3} $, respectively. It follows from (\ref{ps:eqn2}) that
\begin{equation}\label{ps:eqn3}
-\Delta_{g} u_{i} = \frac{f_{i}}{e^{2u_{i}}} \cdot e^{2u_{i}} : = K_{i} e^{2u_{i}} \; {\rm in} \; M, \frac{\partial u_{i}}{\partial \nu} + 1 = A_{i} + 1 = \frac{A_{i} + 1}{e^{u_{i}}} e^{u_{i}} : = \sigma_{i} e^{u_{i}} \; {\rm on} \; \partial M, i = 1, 2, 3.
\end{equation}
Note that the sign conditions of $ K_{i} $ are exactly the same as $ f_{i} $ since $ e^{2u_{i}} > 0 $ as always. Similarly, $ \sigma_{i} > 0 $ on $ \partial M $ since $ A_{i} + 1 > 0 $. The functions $ K_{i}, \sigma_{i}, i = 1, 2, 3 $ are desired Gaussian and geodesic curvatures of $ g_{i} = e^{2u_{i}} g $ for the first three cases.

For the last two cases, we choose the model case $ K_{g} = 1, \sigma_{g} = 0 $ that can be achieved by some conformal change, due to Theorem \ref{intro:thm1} again. Consider the following equations
\begin{equation}\label{ps:eqn4}
-\Delta_{g} u_{i} = B_{i} \; {\rm in} \; M, \frac{\partial u_{i}}{\partial \nu} = h_{i} \; {\rm on} \; \partial M, \int_{M} B_{i} \dvol = -\int_{\partial M} h_{i} dS_{g}, i = 4, 5.
\end{equation}
We choose $ h_{4} $ changing sign on $ \partial M $, and $ h_{5} \leqslant 0 $ on $ \partial M $. Furthermore, we may scale $ h_{i}, i = 4, 5 $ so that $ \lvert B_{i} \rvert < 1 $. With the smooth solutions $ u_{i}, i = 4, 5 $ of (\ref{ps:eqn4}), we have
\begin{equation}\label{ps:eqn5}
-\Delta_{g} u_{i} + 1 = B_{i} + 1 =  \frac{B_{i} + 1}{e^{2u_{i}}} \cdot e^{2u_{i}} : = K_{i} e^{2u_{i}} \; {\rm in} \; M, \frac{\partial u_{i}}{\partial \nu} = \frac{h_{i}}{e^{u_{i}}} e^{u_{i}} : = \sigma_{i} e^{u_{i}} \; {\rm on} \; \partial M, i = 4, 5.
\end{equation}
Clearly, $ K_{i} > 0 $ on $ M $, $ \sigma_{4} $ changes sign on $ \partial M $ and $ \sigma_{5} \leqslant 0 $ on $ \partial M $. By (\ref{ps:eqn1}), we conclude that $ K_{i}, \sigma_{i} $ are Gaussian and geodesic curvatures of the metrics $ g_{i} = e^{2u_{i}} g $ for $ i = 4, 5 $, respectively. 

For the last three, we consider the equations
\begin{align*}
-\Delta_{g} u_{6} & = F_{6} \; {\rm in} \; M, \frac{\partial u_{6}}{\partial \nu} = F_{6}' \; {\rm on} \; \partial M, F_{6} \; \text{changes sign}, \\
& \qquad -\frac{3}{2} < F_{6}' < -\frac{1}{2}, \int_{M} F_{6} \dvol = - \int_{\partial M} F_{6}' dS_{g}; \\
-\Delta_{g} u_{7} & = F_{7} \; {\rm in} \; M, \frac{\partial u_{7}}{\partial \nu} = F_{7}' \; {\rm on} \; \partial M, F_{7} < -1, F_{7}' \; \text{changes sign}, \int_{M} F_{7} \dvol = - \int_{\partial M} F_{7}' dS_{g}; \\
-\Delta_{g} u_{8} & = F_{8} \; {\rm in} \; M, \frac{\partial u_{8}}{\partial \nu} = F_{8}' \; {\rm on} \; \partial M, F_{8} \; \text{changes sign}, F_{8}' < - 1, \int_{M} F_{8} \dvol = - \int_{\partial M} F_{8}' dS_{g}.
\end{align*}
It is clear that all choices of $ F_{i}, F_{i}', i = 6, 7, 8 $ are possible and hence result in smooth solutions $ u_{i}, i = 6, 7, 8 $. Applying model case $ K_{g} = 1, \sigma_{g} = 0 $ for $ u_{7} $ and model case $ K_{g} = 0, \sigma_{g} = 1 $ for $ u_{8} $, we conclude that the last three cases hold. The arguments are very similar as above, we omit them.
\end{proof}
\medskip

When $ \chi(M) \leqslant 0 $, we can get similar existence results as above. We will not list all of them, instead we just show an example.
\begin{proposition}\label{ps:prop2}
Let $ (M, \partial M, g) $ be a compact Riemann surface with non-empty smooth boundary, $ \chi(M) = 0 $. There exists a metric $ \tilde{g} \in [g] $ such that $ K_{\tilde{g}} > 0 $ and $ \sigma_{\tilde{g}} < 0 $.
\end{proposition}
\begin{proof}
We consider the PDE
\begin{equation*}
-\Delta_{g} u = b_{1} \; {\rm in} \; M, \frac{\partial u}{\partial \nu} = b_{2} \; {\rm on} \; \partial M, b_{2} < 0 < b_{1}, \int_{M} b_{1} \dvol = -\int_{\partial M} b_{2} dS_{g}.
\end{equation*}
The argument then follows exactly as in Proposition \ref{ps:prop1}.
\end{proof}

\bibliographystyle{plain}
\bibliography{KWGaussGeodesic}

\end{document}